\documentclass[10pt,leqno]{amsart}
\usepackage[utf8]{inputenc}
\usepackage{amstext}
\usepackage{amsthm}
\usepackage{amssymb}
\usepackage{microtype}
\usepackage[bookmarks=false,
 breaklinks=false,pdfborder={0 0 1},backref=false,colorlinks=false]
 {hyperref}
\hypersetup{pdftitle={Tachibana Revisited},
 pdfauthor={Xiaolong Li and Peter Petersen},
 pdfsubject={Tachibana-type rigidity theorems for closed manifolds under weakened partial-sum conditions on the eigenvalues of curvature operators},
 pdfkeywords={Einstein manifolds, curvature operator, Bochner technique, rigidity}}

\makeatletter

\providecommand{\tabularnewline}{\\}

\numberwithin{equation}{section}
\numberwithin{figure}{section}

\@ifundefined{date}{}{\date{}}

\usepackage{needspace}
\newtheorem{theorem}{Theorem}[section]
\newtheorem{proposition}[theorem]{Proposition}
\newtheorem{corollary}[theorem]{Corollary}
\newtheorem{lemma}[theorem]{Lemma}\theoremstyle{remark}
\newtheorem{remark}[theorem]{Remark}\theoremstyle{plain}
\numberwithin{equation}{section}

\DeclareMathOperator{\tr}{tr}
\DeclareMathOperator{\ad}{ad}
\DeclareMathOperator{\rank}{rank}
\DeclareMathOperator{\Id}{Id}

\title[Tachibana Revisited]{Tachibana's Theorem Revisited}

\author[Li]{Xiaolong Li}
\address{Department of Mathematics and Statistics, Auburn University,
Auburn, AL 36849, USA}
\email{xil0005@auburn.edu}
\thanks{Xiaolong Li's research is partially supported by NSF-DMS
\#2553660 and a start-up grant at Auburn University.}

\author[Petersen]{Peter Petersen}
\address{University of California, Los Angeles, 520 Portola Plaza, CA, 90095}
\email{petersen@math.ucla.edu}


\makeatother

\begin{document}
\begin{abstract}
We prove Tachibana-type rigidity theorems for closed manifolds under weakened partial-sum conditions on the eigenvalues of curvature operators. Specifically, an Einstein manifold of dimension $n\geq4$ is locally symmetric if its curvature operator is $\frac{2(n-1)}{3}$-nonnegative. The corresponding thresholds are $\frac{2(n+1)}{3}$ for the primitive curvature operator of a Kähler-Einstein manifold of complex dimension $n\geq2$, and $\frac{2(n+2)}{3}$ for the curvature operator of a quaternionic-Kähler manifold of real dimension $4n\geq8$. 
\end{abstract}

\maketitle

\section{Introduction}

 This paper is focused on rigidity theorems for Einstein manifolds under suitable nonnegativity conditions on curvatures. There are now a plethora of such results some of which are mentioned later in the introduction. We will be concerned with assumptions that relate to the curvature operator. A classical result of Tachibana \cite{Tachibana1974} states that a closed Einstein manifold with nonnegative curvature operator is locally symmetric. The curvature condition was weakened to $\lfloor\frac{n-1}{2}\rfloor$-nonnegativity in dimensions $n\geq5$ by the second author and Wink \cite{PW21}. They subsequently proved analogous results for Kähler-Einstein manifolds of complex dimension $n\geq4$ under $\frac{n+1}{2}$-nonnegativity of the Kähler curvature operator \cite{PW21Crelle}, and for quaternionic-Kähler manifolds of real dimension $4n\geq8$ under $\lfloor\frac{n+1}{2}\rfloor$-nonnegativity of the quaternionic curvature operator \cite{PW22}.

A self-adjoint operator $L:V\rightarrow V$ with eigenvalues $\lambda_{1}\leq\cdots\leq\lambda_{N}$ is \emph{k-nonnegative}, for $1\leq k\leq N$, provided 
\[
\Sigma_{k}(L):=\lambda_{1}+\cdots+\lambda_{\lfloor k\rfloor}+(k-\lfloor k\rfloor)\lambda_{\lfloor k\rfloor+1}\geq0,
\]
where the last term is omitted when $k$ is an integer. It is \emph{$k$-positive} if $\Sigma_{k}(L)>0$. Since the averages $\frac{1}{k}\Sigma_{k}(L)$ are nondecreasing, larger values of $k$ give weaker conditions. Note that if $W\subset V$ is invariant under $L$, then 
\begin{equation}
\Sigma_{k}\left(L\right)\leq\Sigma_{k}\left(L|_{W}\right)\label{eq:restriction_ineq}
\end{equation}
for $1\leq k \leq \dim(W)$. 

The purpose of this paper is to further weaken the curvature operator conditions in Tachibana-type results. We will assume that all manifolds are closed connected Einstein manifolds. The curvature tensor, $R$, is viewed as a self-adjoint operator on $\bigwedge^{2}TM$. 
Given that the curvature operator vanishes on the orthogonal complement of the holonomy $\mathfrak{h}\subset\mathfrak{so}\left(T_{p}M\right)$, we can in the Kähler case restrict the curvature to $\mathfrak{u}\left(n\right)=\mathfrak{u}\left(1\right)\oplus\mathfrak{su}\left(n\right)\subset\mathfrak{so}\left(2n\right)$,
and in the quaternionic-Kähler case to $\mathfrak{sp}\left(1\right)\oplus\mathfrak{sp}\left(n\right)\subset\mathfrak{so}\left(4n\right)$.
As the metric is Einstein we can decompose the curvature tensor: 
$$R=cR_{1}+R_{0},$$
where $c$ is a constant that depends on the Einstein constant, $R_{1}$ is a parallel tensor that corresponds to the curvature tensor of the canonical metric on $S^{n}$, $\mathbb{CP}^{n}$, or $\mathbb{HP}^{n}$, respectively, and $R_{0}$ is an algebraic curvature operator with vanishing Ricci curvature.
In the real case $R_{0}=W$ is the Weyl tensor, in the Kähler case $R_{0}=B$ is the Bochner tensor which vanishes on $\mathfrak{u}\left(1\right)$, and in the quaternionic-Kähler case $R_{0}$ vanishes on $\mathfrak{sp}\left(1\right)$. See also \cite{Besse1987},  \cite{Bochner1949},\cite{Alekseevskii1968},  and \cite{Salamon1982} for more on these decompositions.

To obtain slightly stronger results we will use \eqref{eq:restriction_ineq} to further restrict the curvature operator to $\mathfrak{g}=\mathfrak{so}(n)$, $\mathfrak{su}(n)$, or $\mathfrak{sp}(n)$ depending on which of the three geometries we are considering.

\begin{theorem}\label{thm:einstein} Let $(M^{n},g)$ be a closed, connected Einstein manifold of dimension $n\geq4$. If its curvature operator is $\frac{2(n-1)}{3}$-nonnegative at every point, then $(M,g)$ is locally symmetric. \end{theorem}

Note that the nonnegativity assumption is weaker than that used to control Betti numbers in \cite{PW21}. Appendix \ref{sec:appendix-ricci} has results that address whether it suffices to assume that the curvature tensor is harmonic given that it is $\frac{2(n-1)}{3}$-nonnegative. This is actually the case for $n\leq 341$. For larger dimensions this remains open. However, algebraic examples show that the Bochner technique fails to offer an answer.
\Needspace{9\baselineskip} \begin{theorem}\label{thm:kahler}
Let $(M^{2n},g,J)$ be a closed, connected Kähler-Einstein manifold of complex dimension $n\geq2$. If the curvature operator restricted to $\mathfrak{su}\left(n\right)$ is $\frac{2(n+1)}{3}$-nonnegative at every point, then $(M,g)$ is locally symmetric. \end{theorem}

This result strengthens the conclusions of Theorem D in \cite{BNPSW2026} with weaker nonnegativity assumptions. In contrast to the real case, Appendix \ref{sec:appendix-kahler-harmonic} shows that it suffices to assume that the curvature is harmonic when the metric is Kähler.

\begin{theorem}\label{thm:quaternionic} 
Let $(M^{4n},g,I,J,K)$, $n\geq2$, be a closed, connected quaternionic-Kähler $4n$-manifold. If the quaternionic-Kähler curvature operator restricted to $\mathfrak{sp}\left(n\right)$ is $\frac{2(n+2)}{3}$-nonnegative at every point, then $(M,g)$ is locally symmetric. \end{theorem}

The proofs use the Bochner technique. For an Einstein metric, the curvature tensor is divergence-free and satisfies 
\[
\frac{1}{2}\Delta\|R\|^{2}=\|\nabla R\|^{2}+\frac{1}{2}\sum_{i}\lambda_{i}\|\eta_{i}\cdot R\|^{2},
\]
where $\Delta=-\operatorname{div}\nabla=\nabla^{*}\nabla$, the $\eta_{i}$ form an orthonormal eigenbasis of the curvature operator on the relevant Lie algebra $\mathfrak{g}=\mathfrak{so}\left(n\right),\,\mathfrak{u}\left(n\right),\,\mathfrak{sp}\left(1\right)\oplus\mathfrak{sp}\left(n\right)$,
with $\lambda_{i}$ being the corresponding eigenvalues. The dot denotes the Lie algebra
action of real skew-adjoint operators on tensors. All displayed norms
are Hilbert-Schmidt norms of curvature operators; see Section~\ref{sec:preliminaries}.
This identity is also the point of departure for the arguments in \cite{PW21,PW22}.

Since $cR_{1}$ is parallel, $R_{0}$ is also divergence free and satisfies the second Bianchi identity. The Bochner formula can consequently be reduced to 
\begin{equation}
\frac{1}{2}\Delta\|R_{0}\|^{2}=\|\nabla R_{0}\|^{2}+\frac{1}{2}Q(R,R_{0}), \textrm{ where }Q(R,R_{0})=\sum_{i}\lambda_{i}\|\eta_{i}\cdot R_{0}\|^{2}.\label{eq:bochner}
\end{equation}
See \cite[Proposition~1.4 and Corollary~1.5]{PW22}. The geometric problem then involves showing that $Q(R,R_{0})$ is nonnegative under the nonnegativity assumptions.

The $\eta_{i}$ that act nontrivially on $R_{0}$ are contained in $\mathfrak{g}=$$\mathfrak{so}(n)$, $\mathfrak{su}(n)$, or $\mathfrak{sp}(n)$. Thus these particular $\eta_{i}$ diagonalize $R_{0}$ as well as the restriction $R|_{\mathfrak{g}}$.
We shall henceforth assume that $\eta_{i}$, $i=1,...,N=\dim\mathfrak{g}$ are the eigenvectors for the restriction of $R|_{\mathfrak{g}}$ with eigenvalues $\lambda_{1}\leq\cdots\leq\lambda_{N}$. This means that we can write
\[
Q(R,R_{0})=\sum_{i=1}^N\lambda_{i}\|\eta_{i}\cdot R_{0}\|^{2}
\]
in equation \eqref{eq:bochner}.

Previous results use the general optimal action estimate on algebraic curvature operators: $\|\eta_{i}\cdot R\|^{2}\leq8\|R\|^{2}$. Our main algebraic improvement comes from a better action bound on suitable averages of these action terms: 
\begin{equation}
\sum^{l}_{i=1}\|\eta_{i}\cdot R_{0}\|^{2}\leq6l\|R_{0}\|^{2}, \text{ for }1\leq l\leq\dim\mathfrak{g}.\label{eq:intro-ordered}
\end{equation}
Proposition~\ref{prop:ordered} proves this estimate for arbitrary self-adjoint endomorphisms; neither the first Bianchi identity nor a trace condition is required. Inequality~\eqref{eq:intro-ordered} improves the control of the initial sums by using that $\eta_{i}$ are eigenvectors and, crucially, that the corresponding eigenvalues are ordered.

For the three curvature components, the total action is given by
\[
\sum^{\dim\mathfrak{g}}_{i=1}\|\eta_{i}\cdot R_{0}\|^{2}=4\kappa\|R_{0}\|^{2},\textrm{ where } \kappa=\begin{cases}
n-1, & R_{0}=W,\\
n+1, & R_{0}=B,\\
n+2, & R_{0}=R|_{\mathfrak{sp}\left(n\right)}-c \Id.
\end{cases}
\]
The curvature identities are being used to establish these formulas. Combining the total with \eqref{eq:intro-ordered} by summation by parts gives, for the restriction $R|_{\mathfrak{g}}$, the following inequality
\[
Q(R,R_{0})\geq6\|R_{0}\|^{2}\Sigma_{\frac{2\kappa}{3}}\left(R|_{\mathfrak{g}}\right).
\]
Integration of \eqref{eq:bochner} now gives parallel curvature. The coefficient 6 therefore yields an improvement by a factor of $\frac{4}{3}$ over the threshold $\frac{\kappa}{2}$ furnished by the old action bound of 8 with the same total action. 
In dimension 4, Theorem~\ref{thm:einstein} recovers the known 2-nonnegative Einstein rigidity statement discussed in \cite{PW21}. In the quaternionic case, we verify the total-action coefficient directly and point out the counting correction to \cite[Corollary~4.5]{PW22} in Remark~\ref{rem:quaternionic-normalization}.

The proof of \eqref{eq:intro-ordered} depends on the three Lie algebras as subalgebras of $\mathfrak{so}$. We will identify $\mathfrak{so}\left(V\right)$ with $\bigwedge^{2}V$ and use the inner product structure on $\bigwedge^{2}V$ which is half of the Hilbert-Schmidt for operators in $\mathfrak{so}\left(V\right)$. This allows us to control the structure constants of the Lie algebras and certain restricted sums of squares of structure constants. The details can be found in Section~\ref{sec:ordered-estimates} and do not depend on the individual structure of the three Lie algebras.

The three thresholds are sharp for the pointwise algebraic implication that $Q(R,R_{0})\geq0$. We exhibit a single four-dimensional Ricci-flat curvature block that embeds in all three geometries and produces, for every larger threshold, an admissible algebraic curvature tensor with a positive eigenvalue sum and a negative Bochner term; see Remark~\ref{rem:algebraic-sharpness}. This establishes the sharpness of the algebraic criterion used in the proof.

Tachibana-type rigidity also appears under other curvature conditions and in geometric evolution equations. Micallef and Wang \cite[Theorem~4.4]{MM93} proved that a closed Einstein four-manifold with nonnegative isotropic curvature is locally symmetric; Brendle \cite[Theorem~1]{Brendle2010} established the corresponding result in general. 
Böhm and Wilking \cite[Theorem~1]{BW08} proved that the normalized Ricci flow on a compact manifold with $2$-positive curvature operator converges to a metric of constant sectional curvature. 
In the case of Kähler-Einstein manifolds one can prove similar results using curvature assumptions related to sectional curvature, bisectional curvature and holomorphic sectional curvature (see \cite{Berger1966}, and \cite{Gray1977}). Notably Mok and Zhong in \cite{MokZhong1986} showed that such manifolds are locally symmetric provided the holomorphic bisectional curvature is nonnegative. They also offer an excellent introduction to prior results in this vein.
For quaternionic-Kähler manifolds the LeBrun-Salamon conjecture states that such manifolds are symmetric spaces provided the scalar curvature is positive (see \cite{LeBrunSalamon1994}). This remains an open question. Theorem \ref{thm:quaternionic} offers a partial answer as does the recent result of Brendle and Semmelmann (see \cite{BrendleSemmelmann2025}) where the authors assume that the sectional curvature is nonnegative.

There are further extensions beyond the Einstein and compact settings.
Petersen and Wink \cite[Theorem~B]{PW22} proved rigidity for compact Kähler manifolds with divergence-free Bochner tensor under $\lfloor(n+1)/2\rfloor$-nonnegativity of the Kähler curvature operator. Colombo, Mariani, and Rigoli \cite{ColomboMarianiRigoli2024} proved local symmetry for compact manifolds with harmonic curvature and $\lfloor(n-1)/2\rfloor$-nonnegative curvature operator. 
They also obtained complete-manifold versions under additional analytic hypotheses, including parabolicity, integral control of the Weyl tensor, or stronger positive lower bounds for the relevant averaged curvature. The improved thresholds in the present paper use the Einstein and special-holonomy assumptions.

The paper is organized as follows. Section~\ref{sec:preliminaries} fixes the curvature, action, and norm conventions. Section~\ref{sec:ordered-estimates} proves the action estimate \eqref{eq:intro-ordered}. Section~\ref{sec:applications} establishes the spectral comparison, proves the three theorems, and gives the common algebraic sharpness construction.

\textbf{AI Disclosure:} The authors used OpenAI’s ChatGPT as an assistive tool in the development of the paper. The authors take full responsibility for all mathematical arguments.

\textbf{Acknowledgments.} This work was initiated at the workshop {\em The Bochner Technique} at the American Institute of Mathematics (AIM) in May 2026. We would like to thank AIM for this opportunity and for their hospitality during our stay.

\section{Preliminaries}
\label{sec:preliminaries}

\subsection{Curvature tensors and curvature operators}

Let $V$ be a finite-dimensional real inner product space. An \emph{algebraic curvature tensor} $C$ on $V$ is a covariant $\left(0,4\right)$-tensor with the symmetries 
\[
C(x,y,z,w) =-C(y,x,z,w)=-C(x,y,w,z)=C(z,w,x,y)
\]
and the first Bianchi identity 
\[
C(x,y,z,w)+C(y,z,x,w)+C(z,x,y,w)=0.
\]
The inner product on $V$ induces one on $\bigwedge^{2}V$ by 
\[
\langle x\wedge y,z\wedge w\rangle=\langle x,z\rangle\langle y,w\rangle-\langle x,w\rangle\langle y,z\rangle.
\]
We use the same letter $C$ for the associated curvature operator, defined by 
\[
\langle C(x\wedge y),z\wedge w\rangle=C(x,y,z,w).
\]
The skew and pair symmetries are equivalent to this operator being self-adjoint. The first Bianchi identity is an additional condition.
Our convention makes the curvature operator of the unit sphere the identity, and the Ricci contraction is 
\[
\operatorname{Ric}_{C}(x,y)=\sum_{i}C(x,e_{i},y,e_{i}).
\]

We write $|C|$ for the tensor norm and $\|C\|$ for the Hilbert-Schmidt norm of the associated operator. Thus, in an orthonormal basis $e_{1},\ldots,e_{n}$ of $V$, we have 
\[
\|C\|^{2}=\tr(C^{*}C)=\sum_{a<b}|C(e_{a}\wedge e_{b})|^{2}
\]
and 
\[
|C|^{2}=4\|C\|^{2}. 
\]
We use operator norms throughout the estimates and the Bochner formula.
For an element $L\in\bigwedge^{2}V$ we use $|L|$ for its induced Euclidean norm.

\subsection{Skew-adjoint matrices and the commutator action}

Identify $\bigwedge^{2}V$ with the space $\mathfrak{so}(V)$ of skew-adjoint endomorphisms by 
\[
(x\wedge y)z=\langle x,z\rangle y-\langle y,z\rangle x.
\]
Under this identification, for $L,K\in\mathfrak{so}(V)$, 
\[
\langle L,K\rangle=-\frac{1}{2}\tr_{V}(LK).
\]
A \emph{Lie subalgebra} $\mathfrak{g}\subset\mathfrak{so}(V)$ is simply a vector subspace closed under the Lie bracket $[L,K]=LK-KL$.
For $L\in\mathfrak{g}$ write $\ad_{L}(X)=[L,X]$. Cyclicity of the trace gives 
\[
\begin{aligned}\langle[L,X],Y\rangle & =-\tfrac{1}{2}\tr\bigl((LX-XL)Y\bigr)\\
 & =-\tfrac{1}{2}\tr\bigl(X(YL-LY)\bigr) \\
 & =-\langle X,[L,Y]\rangle.
\end{aligned}
\]
Thus $\ad_{L}$ is skew-adjoint. This identity also shows that $\ad_{L}$ preserves both $\mathfrak{g}$ and its orthogonal complement in $\mathfrak{so}(V)$.

For a self-adjoint endomorphism $T$ of $\mathfrak{g}$, define 
\begin{equation}
L\cdot T=[\ad_{L},T],\quad(L\cdot T)X=[L,TX]-T[L,X].\label{eq:skew-action}
\end{equation}
Since $\ad_{L}$ is skew-adjoint, $L\cdot T$ is self-adjoint. If $T$ comes from a curvature tensor, this agrees with the standard Lie algebra action 
\[
(L\cdot T)(x_{1},x_{2},x_{3},x_{4})=-\sum^{4}_{r=1}T(x_{1},\ldots,Lx_{r},\ldots,x_{4}).
\]
Indeed, the induced action on $\bigwedge^{2}V$ sends $x\wedge y$ to $Lx\wedge y+x\wedge Ly=[L,x\wedge y]$, and moving this skew-adjoint map between the two slots gives \eqref{eq:skew-action}. Extending $T$ by zero on $\mathfrak{g}^{\perp}$ does not change $\|T\|$ or $\|L\cdot T\|$, because both summands are preserved by $\ad_{L}$.

The three algebras used below have concrete matrix descriptions: $\mathfrak{so}(n)$ consists of real skew-symmetric matrices; $\mathfrak{su}(n)$ consists of complex skew-Hermitian matrices of complex trace zero; and $\mathfrak{sp}(n)$ consists of quaternionic skew-Hermitian matrices. In the last two cases, $L^{*}$ denotes conjugate transpose and skew-Hermitian means $L^{*}=-L$. Regarded as real endomorphisms of $\mathbb{C}^{n}$ and $\mathbb{H}^{n}$, their induced inner products are 
\begin{equation}
\langle L,K\rangle=\begin{cases}
-\frac{1}{2}\tr_{\mathbb{R}}(LK), & \mathfrak{g}=\mathfrak{so}(n),\\
-\tr_{\mathbb{C}}(LK), & \mathfrak{g}=\mathfrak{su}(n),\\
-2\operatorname{Re}\tr_{\mathbb{H}}(LK), & \mathfrak{g}=\mathfrak{sp}(n).
\end{cases}\label{eq:trace-metrics}
\end{equation}
All three are restrictions of the same real trace metric. In particular, no rescaling is needed when applying a matrix estimate on $\mathfrak{so}(V)$ to either of the indicated subalgebras.

\subsection{The curvature operators in the holonomy reductions}

The curvature operator is supported on the holonomy algebra: its image is contained in that algebra, and self-adjointness makes it vanish on the orthogonal complement. It is enough to use the following parallel Lie-algebra bundles containing the holonomy algebra, even if the actual holonomy is a proper subalgebra.

For a Kähler manifold $(M^{2n},g,J)$, the real $(1,1)$-forms identify with $\mathfrak{u}(n)$, the skew-adjoint endomorphisms commuting with $J$. The Kähler form $\omega$ spans the center, and the primitive real $(1,1)$-forms are 
\[
\sideset{}{_{0}^{1,1}}\bigwedge=\omega^{\perp}\cap\sideset{}{_{\mathbb{R}}^{1,1}}\bigwedge\cong\mathfrak{su}(n).
\]
The full Kähler curvature operator is $K=R|_{\mathfrak{u}(n)}$. If $\pi_{0}$ is orthogonal projection onto $\sideset{}{_{0}^{1,1}}\bigwedge$, the primitive operator is the self-adjoint compression $\pi_{0} \circ K|_{\Lambda^{1,1}_{0}}$.
For a Kähler-Einstein metric this subspace is invariant, so $\pi_{0} \circ K|_{\Lambda^{1,1}_{0}}=R|_{\mathfrak{su}(n)}$.

For a quaternionic-Kähler manifold of real dimension $4n\geq8$, the holonomy lies in $\operatorname{Sp}(n)\operatorname{Sp}(1)$ and the metric is Einstein. The two commuting factors give orthogonal subalgebras
$\mathfrak{sp}(n)\oplus\mathfrak{sp}(1)\subset\mathfrak{so}(4n)$.
The quaternionic-Kähler curvature operator is $K=R|_{\mathfrak{sp}(n)\oplus\mathfrak{sp}(1)}$, and both summands are invariant. The restrictions and the curvature decompositions in these settings are specified in Section~\ref{sec:applications};
see also \cite[Sections 3 and 4]{PW22}.

\section{The summed action estimate}

\label{sec:ordered-estimates}

We consider a Lie subalgebra $\mathfrak{g}\subseteq\mathfrak{so}(V)$, with the induced inner product and commutator action from Section~\ref{sec:preliminaries}:
\begin{equation}
\langle L,X\rangle=-\frac{1}{2}\tr_{V}(LX),\quad L\cdot T=[\ad_{L},T].\label{eq:induced-lie-metric}
\end{equation}
Thus $|L|$ denotes the norm of an element of $\mathfrak{g}$, whereas $\|T\|$ denotes the Hilbert-Schmidt norm of an endomorphism. By \eqref{eq:trace-metrics}, this setting includes $\mathfrak{so}(n)$, $\mathfrak{su}(n)$, and $\mathfrak{sp}(n)$ with exactly the normalizations needed in the geometric applications.

\begin{proposition}[Ordered eigenvector estimate]\label{prop:ordered}
Let $\mathfrak{g}\subseteq\mathfrak{so}(V)$ carry the inner product \eqref{eq:induced-lie-metric}, and $N=\dim\mathfrak{g}$. Let $T\colon\mathfrak{g}\to\mathfrak{g}$ be any self-adjoint endomorphism.
Choose an orthonormal eigenbasis with 
\[
T\eta_{i}=\lambda_{i}\eta_{i},\textrm{ where } \lambda_{1}\leq\cdots\leq\lambda_{N}.
\]
It follows that 
\begin{equation}
\sum^{l}_{i=1}\|\eta_{i}\cdot T\|^{2}\leq6l\|T\|^{2}, \textrm{ for } 1\leq l\leq N.\label{eq:ordered-estimate}
\end{equation}
In particular, the estimate holds on $\mathfrak{so}(n)$, $\mathfrak{su}(n)$, and $\mathfrak{sp}(n)$ with the metrics above. No trace condition or curvature identity is imposed on $T$. \end{proposition}

The order of the eigenvectors is essential to the argument. We first establish a bracket estimate that will also account for the cases $l=1$ and $l=2$.

\subsection{A bracket estimate for orthonormal families}

We start with an estimate that controls the sum of the squared norms of certain Lie brackets.

\begin{lemma}[Orthonormal bracket estimate]\label{lem:bracket-family}
For every unit $L\in\mathfrak{g}$ and every orthonormal family $X_{1},\ldots,X_{k}\in\mathfrak{g}$,
\begin{equation}
\sum^{k}_{s=1}|[L,X_{s}]|^{2}\leq k+2.\label{eq:bracket-family}
\end{equation}
\end{lemma}

\begin{proof} 
Inclusion into $\mathfrak{so}(V)$ preserves the bracket and the inner product. It is therefore enough to prove the estimate in $\mathfrak{so}(V)$ itself.

Write $n=\dim V$ and $r=\lfloor \frac{n}{2}\rfloor$ and consider a unit $L\in\bigwedge^{2}V$. The eigenvalues of $L$ are purely imaginary and come in conjugate pairs
\[
\pm\sqrt{-1}a_{1},...,\pm\sqrt{-1}a_{r}\textrm{ with }a_{1}\geq\cdots\geq a_{r}\geq0, 
\]
with an additional zero eigenvalue when $n=2r+1$. Our normalization implies 
\[
|L|^{2}=-\frac{1}{2}\tr_{V}(L^{2})=\sum^{r}_{i=1}a^{2}_{i}=1.
\]
Under $\mathfrak{so}(V)=\bigwedge^{2}V$, the bracket acts by 
\[
\ad_{L}(x\wedge y)=Lx\wedge y+x\wedge Ly.
\]
The eigenvalues are given by $\pm\sqrt{-1}\left(a_{i}\pm a_{j}\right)$, $i<j$, and $\pm\sqrt{-1}a_{i}$ when $n$ is odd, and there is in addition an $r$-dimensional subspace in the kernel. The operator $H=-\ad^{2}_{L}$ will have nonnegative eigenvalues $\sigma_{\alpha}$ of the form $\left(a_{i}\pm a_{j}\right)^{2}$, $i\neq j$ and possibly $a^{2}_{i}$. We contend that
\[
\sum_{\alpha}\left(\sigma_{\alpha}-1\right)_{+}\leq2.
\]
Since $\left|L\right|^{2}=1$ we have that $0\leq a_{i}\leq1$. Thus we can discount the eigenvalues $a^{2}_{i}$ and $\left(a_{i}-a_{j}\right)^{2}$.
When taking multiplicities into account we see that
\[
\sum_{\alpha}\left(\sigma_{\alpha}-1\right)_{+}=2\sum_{i<j}\left(\left(a_{i}+a_{j}\right)^{2}-1\right)_{+}.
\]
Only pairs satisfying $a_i+a_j>1$ contribute to the sum.
Any two such pairs must have an index in common, since two disjoint pairs would give
\[
1\ge a_i^2+a_j^2+a_p^2+a_q^2
\ge \frac{(a_i+a_j)^2+(a_p+a_q)^2}{2}>1.
\]
Unless the sum is zero, the contributing pairs therefore either share a common index or consist of the three pairs among three indices. This is immediate if there is only one contributing pair.
Otherwise, relabel two distinct contributing pairs as $\{1,2\}$ and $\{1,3\}$.
Any contributing pair not containing $1$ must be $\{2,3\}$. If this pair occurs, every contributing pair must intersect all three and hence must itself be one of them; otherwise, all contributing
pairs contain $1$. In the first case, relabel the contributing pairs as $\{1,j\}$, $2\le j\le m+1$. Using $\sum_i a_i^2=1$, we obtain
\begin{align*}
    & \ \ \ 1 - \sum_{j=2}^{m+1}\left((a_1+a_j)^2-1\right) \\
    & = (m+1)\sum_{j=1}^{r}a_j^2 - \sum_{j=2}^{m+1}(a_1+a_j)^2\\
    & = (m+1)\left(a_1^2 + \sum_{j=2}^{m+1} a_j^2 +\sum_{j=m+2}^{r} a_j^2 \right) - ma_1^2 -2a_1 \sum_{j=2}^{m+1} a_j - \sum_{j=2}^{m+1} a_j^2 \\
    & = a_1^2 -2a_1 \sum_{j=2}^{m+1} a_j  + m \sum_{j=2}^{m+1} a_j^2 +(m+1)\sum_{j=m+2}^{r} a_j^2 \\
    & = \left(a_1 - \sum_{j=2}^{m+1} a_j \right)^2 + \sum_{2 \leq i < j \leq m+1} (a_i-a_j)^2+(m+1)\sum_{j=m+2}^{r} a_j^2 \\
    & \geq 0. 
\end{align*}
In the second case, relabel the three indices as $1,2,3$.
The same normalization gives
\[
1-\sum_{1\le i<j\le 3}\bigl((a_i+a_j)^2-1\bigr)
=\sum_{1\le i<j\le 3}(a_i-a_j)^2
  +4\sum_{j=4}^r a_j^2
\ge 0.
\]
Thus, in either case,
\[
\sum_\alpha(\sigma_\alpha-1)_+
=2\sum_{i<j}\bigl((a_i+a_j)^2-1\bigr)_+
\le 2.
\]

We can rewrite the inequality as
\begin{equation*}
\tr(H-I)_{+}\leq2,\label{eq:adjoint-spectral-excess}
\end{equation*}
where $(H-I)_{+}$ is obtained by replacing each negative eigenvalue of $H-I$ by zero. 

Let $P$ be the orthogonal projection onto the span of $X_{1},\ldots,X_{k}$.
The inequalities $H\leq I+(H-I)_{+}$ and $0\leq P\leq I$ imply 
\[
\begin{aligned}\sum^{k}_{s=1}|[L,X_{s}]|^{2} & =\sum^{k}_{s=1}\langle HX_{s},X_{s}\rangle=\tr(PH)\\
 & \leq k+\tr\bigl(P(H-I)_{+}\bigr)\leq k+\tr(H-I)_{+}\leq k+2.
\end{aligned}
\]
The second inequality just sums a positive semidefinite operator over part of an orthonormal basis; it does not require $P$ to commute with $H$. This proves \eqref{eq:bracket-family}. \end{proof}

\subsection{Two consequences of the bracket estimate}

This lemma gives us further control over $\ad_{L}$.

\begin{lemma}\label{lem:pair-and-offdiagonal} For unit vectors $L,X\in\mathfrak{g}$,
\begin{equation}
|\ad_{L}X|^{2}=|[L,X]|^{2}\leq2.\label{eq:pair-bracket}
\end{equation}
If a subspace $H\subseteq\mathfrak{g}$ of dimension $h$ contains $L$, then 
\begin{equation}
\bigl\| P_{H^{\perp}}\ad_{L}|_{H}\bigr\|^{2}\leq h,\label{eq:offdiagonal-bracket}
\end{equation}
where $P_{H^{\perp}}$ is orthogonal projection onto $H^{\perp}$.
\end{lemma}

\begin{proof}
For \eqref{eq:pair-bracket} assume that $\ad_{L}X\neq0$
and consider 
\[
Y=\frac{\ad_{L}X}{\left|\ad_{L}X\right|}.
\]
This is a unit vector orthogonal to $X$, because $\ad_{L}$ is skew-adjoint.
The same property gives 
\[
|\ad_{L}Y|\geq|\langle\ad_{L}Y,X\rangle|=|\langle Y,\ad_{L}X\rangle|=|\ad_{L}X|.
\]
Applying \eqref{eq:bracket-family} to the orthonormal pair $\{X,Y\}$
yields 
\[
2|\ad_{L}X|^{2}\leq|\ad_{L}X|^{2}+|\ad_{L}Y|^{2}\leq4.
\]

For \eqref{eq:offdiagonal-bracket}, write 
\[
B=P_{H^{\perp}}\ad_{L}|_{H}\colon H\longrightarrow H^{\perp}.
\]
In a basis adapted to $H\oplus H^{\perp}$, this is an off-diagonal
block of $\ad_{L}$, and the opposite block is $-B^{*}=P_{H}\ad_{L}|_{H^{\perp}}$.
Since $BL=[L,L]=0$, we have
\[
r=\rank B=\rank B^{*}\leq h-1.
\]
The claim is immediate if $r=0$. Otherwise choose orthonormal bases
$X_{1},\ldots,X_{r}$ of $(\ker B)^{\perp}\subseteq H$ and $Y_{1},\ldots,Y_{r}$
of $\operatorname{im}B\subseteq H^{\perp}$. The two collections together
form an orthonormal family. Therefore 
\[
\begin{aligned}2\|B\|^{2} & =\sum^{r}_{s=1}\bigl(|B X_{s}|^{2}+|B^{*} Y_{s}|^{2}\bigr)\\
 & \leq\sum^{r}_{s=1}\bigl(|[L,X_{s}]|^{2}+|[L,Y_{s}]|^{2}\bigr)\\
 & \leq2r+2\leq2h,
\end{aligned}
\]
where the last line uses \eqref{eq:bracket-family}. \end{proof}

\subsection{The ordered estimate}

\begin{proof}[Proof of Proposition~\ref{prop:ordered}]
Let $V$ be a real Euclidean space and $T$ a self-adjoint operator on $\bigwedge^{2}V$. We select an ordered orthonormal eigenbasis of $T$, $T\left(\eta_{i}\right)=\lambda_{i}\eta_{i}$, $\lambda_{1}\leq\lambda_{2}\leq\cdots\leq\lambda_{N}$.

Fix $1\leq l\leq N$ and define the following partial sums of structure constants:
\begin{equation}
K_{ab}=\sum^{l}_{i=1}\langle[\eta_{i},\eta_{a}],\eta_{b}\rangle^{2},\qquad d_{a}=\sum^{N}_{b=1}K_{ab}=\sum^{l}_{i=1}|[\eta_{i},\eta_{a}]|^{2}.\label{eq:ordered-weights}
\end{equation}
As $\ad_{\eta_{i}}$ is skew-adjoint we obtain the symmetries $K_{ab}=K_{ba}$ and $K_{aa}=0$. 

\medskip{}
\noindent\emph{Step 1: row sums and off diagonal sums.} The goal is to estimate the sum of the row elements in $K_{ab}$ as well as the sum of the entries in an off diagonal block defined by the cut $a\leq h<b$.
Set
\begin{equation}
\alpha=\max\{l,2l-2\},\qquad\beta=\min\{2l,l+2\}.\label{eq:ordered-alpha-beta}
\end{equation}
For $l=1$ these constants are $\alpha=1$, $\beta=2$; for $l\geq2$ they are $\alpha=2l-2$, $\beta=l+2$. In particular, 
\begin{equation}
\alpha+\beta=3l.\label{eq:ordered-alpha-beta-sum}
\end{equation}
When $l=1$ the pair estimate \eqref{eq:pair-bracket} shows that $d_{a}\leq2$ and \eqref{eq:bracket-family} with $L=\eta_{a}$ and the orthonormal family $\eta_{1},\ldots,\eta_{j}$ implies $d_{a}\leq l+2$.
Hence 
\begin{equation*}
d_{a}\leq\beta, \text{ for } 1\leq a\leq N.\label{eq:ordered-row-bound}
\end{equation*}

For $1\leq h\leq N$, define the sum of the off-diagonal entries cut off by $a\leq h<b$
\[
C_{h}=\sum_{a\leq h<b}K_{ab}.
\]
We claim that 
\begin{equation}
C_{h}\leq\alpha h, \text{ for }1\leq h\leq N.\label{eq:ordered-cut-bound}
\end{equation}
If $h<l$, every $a\leq h$ is among the selected indices $1,\ldots,l$. Its self-bracket vanishes, so 
\[
C_{h}\leq\sum^{h}_{a=1}d_{a}=\sum^{h}_{a=1}\sum^{l}_{i=1,\,i\neq a}|[\eta_{i},\eta_{a}]|^{2}\leq 2h(l-1)\leq\alpha h.
\]
If $h\geq l$, put $H_{h}=\operatorname{span}\{\eta_{1},\ldots,\eta_{h}\}$.
Each selected vector $\eta_{i}$, $i\leq l$, belongs to $H_{h}$.
The rectangular matrix of $P_{H^{\perp}_{h}}\ad_{\eta_{i}}|_{H_{h}}$ has entries $\langle[\eta_{i},\eta_{a}],\eta_{b}\rangle$ with $a\leq h<b$.
Thus \eqref{eq:offdiagonal-bracket} gives 
\[
C_{h}=\sum^{l}_{i=1}\bigl\| P_{H^{\perp}_{h}}\ad_{\eta_{i}}|_{H_{h}}\bigr\|^{2}\leq lh\leq\alpha h.
\]
This proves \eqref{eq:ordered-cut-bound} in both cases.

\medskip{}
\noindent\emph{Step 2: write the commutator norms as a weighted sum.}
For a real vector $x=(x_{1},\ldots,x_{N})$, define 
\[
E(x)=\sum_{a<b}K_{ab}(x_{a}-x_{b})^{2}.
\]
Since $T$ is diagonal in our basis, 
\[
\langle[\ad_{\eta_{i}},T]\eta_{a},\eta_{b}\rangle=(\lambda_{a}-\lambda_{b})\langle[\eta_{i},\eta_{a}],\eta_{b}\rangle.
\]
Summing the squares of all matrix entries, and then using $K_{ab}=K_{ba}$, gives, with $\lambda=(\lambda_{1},\ldots,\lambda_{N})$, 
\begin{equation}
\sum^{l}_{i=1}\|\eta_{i}\cdot T\|^{2}=\sum_{a,b}K_{ab}(\lambda_{a}-\lambda_{b})^{2}=2E(\lambda).\label{eq:ordered-energy}
\end{equation}
The factor 2 counts the entries on both sides of the matrix diagonal.

\medskip{}
\noindent\emph{Step 3: estimate the positive and negative parts.}
We write each eigenvalue as the difference of its positive and negative parts 
\[
\lambda_{a}=\lambda^{+}_{a}-\lambda^{-}_{a},\quad\lambda^{-}_{a}=(-\lambda_{a})_{+},\quad\lambda^{+}_{a}=(\lambda_{a})_{+}.
\]
Clearly
\begin{equation}
\sum_{i}\left(\lambda^{+}_{i}\right)^{2}+\sum_{i}\left(\lambda^{-}_{i}\right)^{2}=\sum_{i}\lambda^{2}_{i}=\|T\|^{2}.\label{eq:ordered-parts-norm}
\end{equation}
For $\lambda^{+}=\left(\lambda^{+}_{1},...,\lambda^{+}_{N}\right)$, the elementary inequality $(\lambda^{+}_{a}-\lambda^{+}_{b})^{2}\leq\left(\lambda^{+}_{a}\right)^{2}+\left(\lambda^{+}_{b}\right)^{2}$
and the row bound \eqref{eq:ordered-row-bound} give 
\begin{equation}
E(\lambda^{+})\leq\sum_{a<b}K_{ab}\left(\left(\lambda^{+}_{a}\right)^{2}+\left(\lambda^{+}_{b}\right)^{2}\right)=\sum_{a}d_{a}\left(\lambda^{+}_{a}\right)^{2}\leq\beta|\lambda^{+}|^{2}.\label{eq:ordered-positive-part}
\end{equation}

For $\lambda^{-}=\left(\lambda^{-}_{1},...,\lambda^{-}_{N}\right)$, the ordering of the eigenvalues gives the additional information $\lambda^{-}_{1}\geq\cdots\geq\lambda^{-}_{N}\geq0$.
If $a<b$, then 
\[
(\lambda^{-}_{a}-\lambda^{-}_{b})^{2}\leq\left(\lambda^{-}_{a}\right)^{2}-\left(\lambda^{-}_{b}\right)^{2}=\sum^{b-1}_{i=a}\left(\left(\lambda^{-}_{i}\right)^{2}-\left(\lambda^{-}_{i+1}\right)^{2}\right).
\]
Let $\lambda^{-}_{N+1}=0$. Reversing the order of summation expresses the resulting bound in terms of the cuts $C_{h}$: 
\[
\begin{aligned}E(\lambda^{-}) & \leq\sum_{a<b}K_{ab}\left(\left(\lambda^{-}_{a}\right)^{2}-\left(\lambda^{-}_{b}\right)^{2}\right)\\
& =\sum^{N}_{h=1}\left(\left(\lambda^{-}_{h}\right)^{2}-\left(\lambda^{-}_{h+1}\right)^{2}\right)\sum_{a \leq h < b} K_{ab}\\
 & =\sum^{N}_{h=1}\left(\left(\lambda^{-}_{h}\right)^{2}-\left(\lambda^{-}_{h+1}\right)^{2}\right)C_{h}\\
 & \leq\alpha\sum^{N}_{h=1}h\left(\left(\lambda^{-}_{h}\right)^{2}-\left(\lambda^{-}_{h+1}\right)^{2}\right) \\
 &=\alpha\sum^{N}_{i=1}\left(\lambda^{-}_{i}\right)^{2}.
\end{aligned}
\]
Here all coefficients $\left(\lambda^{-}_{h}\right)^{2}-\left(\lambda^{-}_{h+1}\right)^{2}$ are nonnegative, so \eqref{eq:ordered-cut-bound} applies term by term. The last equality is a telescoping sum. We have proved 
\begin{equation}
E(\lambda^{-})\leq\alpha|\lambda^{-}|^{2}.\label{eq:ordered-negative-part}
\end{equation}

\medskip{}
\noindent\emph{Step 4: combine the two estimates.} 
The function $E^{1/2}$ is a semi-norm: it is the Euclidean norm of the vector 
\[
\bigl(\sqrt{K_{ab}}(x_{a}-x_{b})\bigr)_{a<b},
\]
which depends linearly on $x$. The triangle inequality for this semi-norm, followed
by Cauchy-Schwarz, gives 
\[
\begin{aligned}E\left(\lambda\right) & \leq\bigl(\sqrt{E(\lambda^{-})}+\sqrt{E(\lambda^{+})}\bigr)^{2}\\
 & \leq\bigl(\sqrt{\alpha}\,|\lambda^{-}|+\sqrt{\beta}\,|\lambda^{+}|\bigr)^{2}\\
 & \leq(\alpha+\beta)(|\lambda^{-}|^{2}+|\lambda^{+}|^{2})\\
 &=3l\|T\|^{2}.
\end{aligned}
\]
Together with \eqref{eq:ordered-energy}, this is \eqref{eq:ordered-estimate}.
\end{proof}

\section{Geometric applications and algebraic sharpness}

\label{sec:applications}

We apply Proposition~\ref{prop:ordered} to the Weyl, Bochner, and quaternionic Weyl operators. The ordered estimate is common to all three settings. The different geometric thresholds come from their total-action identities. We establish the spectral comparison and the integration argument in a unified way. The three theorems are immediate consequences of this one result. A single four-dimensional curvature block also gives the algebraic sharpness examples in all three settings.

\subsection{Curvature components and normalizations}

The notation used throughout this section is summarized below. Here $\mathfrak{g}$ is the parallel Lie-algebra bundle on which the trace-free curvature component $R_0$ is supported, and $N=\dim\mathfrak{g}$.
\begin{center}
\global\long\def\arraystretch{1.2}%
\begin{tabular}{lcccc}
\hline 
Geometry  & $\mathfrak{g}$  & $R_{0}$  & $N$  & $\kappa$\tabularnewline
\hline 
Einstein  & $\mathfrak{so}(n)$  & $W$  & $n(n-1)/2$  & $n-1$\tabularnewline
Kähler-Einstein  & $\mathfrak{su}(n)$  & $B$  & $n^{2}-1$  & $n+1$\tabularnewline
Quaternionic-Kähler  & $\mathfrak{sp}(n)$  & $R|_{\mathfrak{sp}\left(n\right)}-c\Id$ & $n(2n+1)$  & $n+2$\tabularnewline
\hline 
\end{tabular}
\par\end{center}

The assumptions are $n\geq4$ in the first row and $n\geq2$ in the other two rows. All norms and Lie-algebra inner products are those of Section~\ref{sec:preliminaries}.

For an Einstein metric the curvature decomposition in the real case is 
\begin{equation}
R=c\operatorname{Id}+W.\label{eq:einstein-decomposition}
\end{equation}
The identity $\operatorname{Id}$ on $\bigwedge^{2}TM$ is the curvature
operator of constant sectional curvature one, and $W$ is the Weyl
operator.

In the Kähler-Einstein case we choose the parallel tensor $R_{1}$ so that its restriction to $\mathfrak{su}\left(n\right)$ is $\operatorname{Id}$.
Thus
\begin{equation}
\begin{gathered}R=c R_{1}+B,\\
R|_{\mathfrak{su}\left(n\right)}=c\operatorname{Id}+B.
\end{gathered}
\label{eq:kahler-decomposition}
\end{equation}
The Bochner operator $B$ is trace-free on $\mathfrak{su}(n)$ and vanishes on $\mathfrak{u}(1)$. Elements of the center, $\mathfrak{u}(1)$, act trivially on $B$. See the Kähler curvature decomposition in \cite[Section~3]{PW22}.

For a quaternionic-Kähler metric we again choose $R_{1}$ so its restriction to $\mathfrak{sp}\left(n\right)$ is $\operatorname{Id}$. This gives us
\begin{equation}
\begin{gathered}R=cR_{1}+R_{0},\\
R|_{\mathfrak{sp}\left(n\right)}=c\operatorname{Id}+R_{0}.
\end{gathered}
\label{eq:quaternionic-decomposition}
\end{equation}
The quaternionic Weyl operator $R_{0}$ is trace-free on $\mathfrak{sp}(n)$
and vanishes on $\mathfrak{sp}(1)$; see \cite[Section~4]{PW22}.

Thus in all three cases 
\begin{equation}
R=c\operatorname{R_1}+R_{0},\quad R\big|_{\mathfrak{g}}=c\operatorname{Id}+R_{0}.\label{eq:common-curvature-decomposition}
\end{equation}
The scalar $c$ is constant, $\operatorname{Id}$ is parallel and
invariant under the relevant holonomy algebra, and $\nabla R=\nabla R_{0}$.
Since these metrics are Einstein, $R$ is divergence-free. Subtracting
the parallel tensor $c\operatorname{Id}$ shows that $R_{0}$ is also
divergence-free and satisfies the second Bianchi identity. The Bochner
formula~\eqref{eq:bochner} therefore applies to $R_{0}$.

\subsection{A New Weight Principle}

We first record the summation-by-parts statement that turns the ordered
action estimate into a lower bound for the Bochner term. For an ordered
list $\lambda_{1}\leq\cdots\leq\lambda_{N}$, the notation $\Sigma_{k}(\lambda)$
has the same fractional-sum meaning as $\Sigma_{k}(A)$ in the introduction.

\begin{lemma}[Comparison of weighted spectral sums] \label{lem:prefix-comparison}
Let $\lambda_{1}\leq\cdots\leq\lambda_{N}$ and $w_{i}\geq0$. Suppose
that $a\geq0$ and $1\leq k<N$ satisfy 
\[
\sum^{l}_{i=1}w_{i}\leq al, \text{ for } 1\leq l\leq N, \text{ and } \sum^{N}_{i=1}w_{i}=ak.
\]
Then 
\[
\sum^{N}_{i=1}\lambda_{i}w_{i}\geq a\Sigma_{k}(\lambda).
\]
\end{lemma}

\begin{proof} Put $r=\lfloor k\rfloor$ and define the balanced comparison
weights
\[
v_{i}=\begin{cases}
a, & i\leq r,\\
a(k-r), & i=r+1,\\
0, & i\geq r+2.
\end{cases}
\]
The comparison weights have the same total $ak$ as the $w_{i}$.
For $F_{l}=\sum^{l}_{i=1}(v_{i}-w_{i})$, we have $F_{0}=F_{N}=0$
and $F_{l}\geq0$.  When $l\leq r$ this is the assumed initial-sum
bound, and when $l\geq r+1$ it follows from equality of the totals
and $w_{i}\geq0$. Summation by parts gives 
\[
\sum^{N}_{i=1}\lambda_{i}(w_{i}-v_{i})=\sum^{N-1}_{l=1}(\lambda_{l+1}-\lambda_{l})F_{l}\geq0.
\]
Since $\sum_{i}\lambda_{i}v_{i}=a\Sigma_{k}(\lambda)$, the conclusion
follows. \end{proof}

A further trivial consequence of ordering the eigenvalues will be
used below. The averages $\Sigma_{k}(\lambda)/k$ are nondecreasing
in $k$. In particular, if $k<N$, $\sum_{i}\lambda_{i}=0$, and $\Sigma_{k}(\lambda)\geq0$,
then every $\lambda_{i}$ is zero: a proper initial average can equal
the full average only if the entire ordered list is constant. 

\subsection{A common curvature block and the total action}

The next construction will be used twice: to check the quaternionic
total-action coefficient and to give the common sharpness example.

\begin{lemma}[A four-dimensional curvature block] \label{lem:common-curvature-block}
For each of the three Lie algebras in the table there is an admissible
trace-free curvature component $R_{*}$ and an orthonormal triple
$E_{1},E_{2},E_{3}\in\mathfrak{g}$ such that 
\begin{equation}
R_{*}E_{1}=-E_{1},\quad R_{*}E_{2}=-E_{2},\quad R_{*}E_{3}=2E_{3}, \text{ and } R_{*}\big|_{\operatorname{span}\{E_{1},E_{2},E_{3}\}^{\perp}}=0.\label{eq:common-model-definition}
\end{equation}
Here admissible means Weyl in the real case, Bochner in the complex
case, and quaternionic Weyl in the quaternionic case. Its norm and
the three supporting action norms are 
\begin{equation}
\|R_{*}\|^{2}=6,\quad\|E_{1}\cdot R_{*}\|^{2}=\|E_{2}\cdot R_{*}\|^{2}=36,\quad\|E_{3}\cdot R_{*}\|^{2}=0.\label{eq:common-model-support-action}
\end{equation}
For an orthonormal basis $\{\eta_{i}\}$ of $\mathfrak{g}$ extending
this triple and diagonalizing $R_{*}$, with eigenvalues $\mu_{i}$,
one has 
\begin{equation}
\sum_{i}\|\eta_{i}\cdot R_{*}\|^{2}=24\kappa, \text{ and } \sum_{i}\mu_{i}\|\eta_{i}\cdot R_{*}\|^{2}=-72.\label{eq:common-model-total-action}
\end{equation}
\end{lemma}

\begin{proof} The curvature operator is an operator on $\bigwedge^{2}V$,
we split $V=\mathbb{H}\oplus W$ , where $\mathbb{H}$ is a real 4-dimensional
subspace identified with the quaternions. For the complex and quaternionic
settings, take it to be a complex two-plane and a quaternionic line, respectively;
the complex and quaternionic structures act by right multiplication.
On this subspace let $E_{1},E_{2},E_{3}$ be left multiplication by
$\mathrm{i},\mathrm{j},\mathrm{k}$, divided by $\sqrt{2}$, and extend
them by zero. In terms of 4-dimensional geometry $E_{1},E_{2},E_{3}$
span the self-dual component in the splitting $\bigwedge^{2}\mathbb{H}=\bigwedge^{2}_{+}\mathbb{H}\oplus\bigwedge^{2}_{-}\mathbb{H}$.
They are orthonormal for $-\tfrac{1}{2}\operatorname{tr}_{\mathbb{R}}$
and satisfy 
\[
[E_{1},E_{2}]=\sqrt{2}E_{3},\quad\text{etc., cyclically},\text{ and } E^{2}_{i}=-\tfrac{1}{2}\operatorname{Id}\quad\text{on }\mathbb{H}.
\]

Define $R_{*}$ by~\eqref{eq:common-model-definition}, or equivalently
by $R_{*}=-E_{1}\otimes E_{1}-E_{2}\otimes E_{2}+2E_{3}\otimes E_{3}$,
where $(E\otimes E)X=\langle E,X\rangle E$. To verify the curvature
identities, orient $\mathbb{H}$ so that the corresponding 2-forms
$E_{i}$ are self-dual. The $E_{i}\wedge E_{i}$ are then the same
volume form for each $i$. The Bianchi identity for $\sum_{i}\mu_{i}E_{i}\otimes E_{i}$
vanishes because $\mu_{1}+\mu_{2}+\mu_{3}=0$. The Ricci contraction
also vanishes: the Ricci endomorphism of $E_{i}\otimes E_{i}$ is
$-E^{2}_{i}$, so on the supporting 4-plane it is 
\[
\operatorname{Ric}_{R_{*}}=-\sum^{3}_{i=1}\mu_{i}E^{2}_{i}=\tfrac{1}{2}(\mu_{1}+\mu_{2}+\mu_{3})\operatorname{Id}=0.
\]
Extending the tensor by zero preserves both identities. Thus $R_{*}$
is trace free and hence a Weyl tensor in every ambient real dimension.

Left multiplication commutes with right multiplication. In the complex
setting the $E_{i}$ commute with the complex structure and have complex
trace zero, since each is a scalar multiple of a commutator of the
other two. They are therefore primitive real $(1,1)$-forms, and $R_{*}$
is a Ricci-flat Kähler curvature tensor, hence a Bochner tensor. In
the quaternionic setting they lie in $\mathfrak{sp}(n)$ and commute
with all three quaternionic structures. The same Ricci-flat algebraic
curvature tensor is consequently a quaternionic Weyl tensor. These
assertions also follow from the curvature decompositions in \cite[Sections 3 and 4]{PW22}.

The norm $\|R_{*}\|^{2}=1+1+4=6$ is immediate. The adjoint action
of $E_{1}$ rotates the $E_{2},E_{3}$ plane with magnitude $\sqrt{2}$,
so 
\[
\|E_{1}\cdot R_{*}\|^{2}=2(\sqrt{2})^{2}((-1)-2)^{2}=36.
\]
The same calculation gives the action of $E_{2}$, while $E_{3}$
rotates the $-1$-eigenspace $\mathrm{span}\left(E_{1},E_{2}\right)$
and acts trivially on $R_{*}$.

We can count the remaining action norms in all three settings at once.
Regard the support as $\mathbb{R}^{4}$, $\mathbb{C}^{2}$, or $\mathbb{H}^{1}$,
and let $q$ be the dimension of $W$ over the corresponding field/algebra.
Thus 
\[
q=n-4,\quad n-2,\quad n-1,\quad\text{respectively, and in every case }q=\kappa-3.
\]
We split $\bigwedge^{2}V=\bigwedge^{2}_{+}\mathbb{H}\oplus\mathcal{M}\oplus\mathcal{C}$,
where $\mathcal{M}$ consists of mixed skew-adjoint matrices between
$\mathrm{span}\left\{E_{1},E_{2},E_{3}\right\}$ and $W$ and has real dimension
$4q$. $\mathcal{C}$ is the orthogonal complement to $\bigwedge^{2}_{+}\mathbb{H}\oplus\mathcal{M}$.
In block diagonal form we have 
\[
X_{F}=\begin{pmatrix}0 & -F^{*}\\
F & 0
\end{pmatrix}\in\mathcal{M}\textrm{ and }E_{i}=\begin{pmatrix}E_{i}|_{\mathbb{H}} & 0\\
0 & 0
\end{pmatrix}.
\]
The commutator $[E_{i},X_{F}]$ is the mixed matrix corresponding
to $-FE_{i}$. Since $E^{*}_{i}E_{i}=\tfrac{1}{2}\operatorname{Id}$,
the induced real matrix norms give 
\begin{equation}
|[X,E_{i}]|^{2}=\tfrac{1}{2}|X|^{2}, \text{ for } X\in\mathcal{M},\ i=1,2,3.\label{eq:common-model-mixed-bracket}
\end{equation}
These brackets lie outside $\operatorname{span}\{E_{1},E_{2},E_{3}\}$.
The two off-diagonal blocks of $[\operatorname{ad}_{X},R_{*}]$ consequently
have equal squared norms, and for every unit $X\in\mathcal{M}$, 
\[
\|X\cdot R_{*}\|^{2}=2\sum^{3}_{i=1}\mu^{2}_{i}|[X,E_{i}]|^{2}=6.
\]
The elements of $\mathcal{C}$ commute with $\operatorname{span}\{E_{1},E_{2},E_{3}\}$
and act trivially on $R_{*}$. For example, the $\bigwedge^{2}_{-}\mathbb{H}$
factor of $\bigwedge^{2}\mathbb{H}$ commutes with $\operatorname{span}\{E_{1},E_{2},E_{3}\}=\bigwedge^{2}_{+}\mathbb{H}$.
Therefore 
\[
\sum_{i}\|\eta_{i}\cdot R_{*}\|^{2}=36+36+4q\cdot6=24(q+3)=24\kappa.
\]
All elements of $\mathcal{M}\oplus\mathcal{C}$ have $R_{*}$-eigenvalue
zero, while the positive eigenvector $E_{3}$ has zero action. Hence
the eigenvalue-weighted sum is $-36-36=-72$, as asserted. \end{proof}

\begin{proposition}[Total action of the geometric curvature components]
\label{prop:total-action} For the curvature component $R_{0}$ and
constant $\kappa$ in the table, and any orthonormal basis $\eta_{1},\ldots,\eta_{N}$
of $\mathfrak{g}$, 
\begin{equation}
\sum^{N}_{i=1}\|\eta_{i}\cdot R_{0}\|^{2}=4\kappa\|R_{0}\|^{2}.\label{eq:total-action}
\end{equation}
Unlike Proposition~\ref{prop:ordered}, this identity is asserted
for the indicated geometric curvature tensors, not for arbitrary self-adjoint
endomorphisms of $\mathfrak{g}$. \end{proposition}

\begin{proof} For the Riemannian Weyl operator this is \cite[Proposition~2.5(b)]{PW21};
the first Bianchi identity and zero Ricci contraction are essential
hypotheses. For the Bochner operator it is the identity used in \cite[Corollary~3.3]{PW22}.
The center of $\mathfrak{u}(n)$ acts trivially on $B$, so restricting
the sum to $\mathfrak{su}(n)$ removes only a zero term. Passing between
full tensor norms and operator norms multiplies both sides by the
same factor and does not change either coefficient.

For the quaternionic case, we verify the coefficient directly. The
space of trace free algebraic curvature tensors on $\mathfrak{sp}\left(n\right)$
is an irreducible real $\operatorname{Sp}(n)$ module on which $\operatorname{Sp}(1)$
acts trivially; see \cite[Section~4]{PW22}. The quadratic form $R_{0}\mapsto\sum_{i}\|\eta_{i}\cdot R_{0}\|^{2}$
is invariant and independent of the chosen orthonormal basis. The
self-adjoint operator representing this quadratic form commutes with
the group action, so each of its eigenspaces is invariant. Irreducibility
forces a single eigenvalue; the quadratic form is therefore
multiple of $\|R_{0}\|^{2}$. Evaluating on the admissible nonzero
block of Lemma~\ref{lem:common-curvature-block} gives 
\[
\frac{\sum_{i}\|\eta_{i}\cdot R_{*}\|^{2}}{\|R_{*}\|^{2}}=\frac{24(n+2)}{6}=4(n+2).
\]
This proves the quaternionic assertion and completes the proof. \end{proof}

\begin{remark}[The quaternionic normalization] \label{rem:quaternionic-normalization}
Corollary~4.5 of \cite{PW22} lists the coefficient $\tfrac{4}{3}(3n+4)$
in place of $4(n+2)$, the discrepancy is due to minor combinatorial error. 
\end{remark}

\subsection{The Bochner argument and the three geometric theorems}

Theorems \ref{thm:einstein}, \ref{thm:kahler}, and \ref{thm:quaternionic}
from the introduction can now be proven with a unified proof.

\begin{lemma}[Common Bochner comparison] \label{lem:common-bochner-comparison}
In any row of the table, set $k=\frac{2\kappa}{3}$ and let $R|_{\mathfrak{g}}=c\operatorname{Id}+R_0$
be the operator in \eqref{eq:common-curvature-decomposition}. Then
\begin{equation}
Q(R,R_{0})\geq6\|R_{0}\|^{2}\Sigma_{k}(R|_{\mathfrak{g}}).\label{eq:bochner-term-lower-bound}
\end{equation}
If the manifold is closed and connected and $\Sigma_{k}(R|_{\mathfrak{g}})\geq0$
everywhere, then $\nabla R=0$. If additionally $\Sigma_{k}(R|_{\mathfrak{g}})>0$
at one point, then $R=cR_1$ with $c>0$. \end{lemma}

\begin{proof} Choose a pointwise ordered orthonormal eigenbasis of
$R_{0}$ on $\mathfrak{g}$, with eigenvalues $\mu_{1}\leq\cdots\leq\mu_{N}$.
It is also an ordered eigenbasis of $R|_{\mathfrak{g}}$, with eigenvalues
$\lambda_{i}=c+\mu_{i}$. Proposition~\ref{prop:ordered} and \eqref{eq:total-action}
give 
\[
\sum^{l}_{i=1}\|\eta_{i}\cdot R_{0}\|^{2}\leq6l\|R_{0}\|^{2}, \text{ for } 1\leq l\leq N,
\]
and 
\[
\sum^{N}_{i=1}\|\eta_{i}\cdot R_{0}\|^{2}=4\kappa\|R_{0}\|^{2}=6k\|R_{0}\|^{2}, \text{ for } 1\leq k<N.
\]
The additional central direction in the Kähler case
and the $\mathfrak{sp}(1)$ directions in the quaternionic case act
trivially on $T$. Thus the full curvature term in~\eqref{eq:bochner}
is 
\[
Q(R,R_{0})=\sum^{N}_{i=1}\lambda_{i}\|\eta_{i}\cdot R_{0}\|^{2}.
\]
Lemma~\ref{lem:prefix-comparison}, applied with $a=6\|R_{0}\|^{2}$,
proves \eqref{eq:bochner-term-lower-bound}.

Assuming $\Sigma_{k}(R|_{\mathfrak{g}})\geq0$, the Bochner identity~\eqref{eq:bochner}
and integration give 
\[
0=\int_{M}\left(\|\nabla R_{0}\|^{2}+\tfrac{1}{2}Q(R,R_{0})\right)\,d\operatorname{vol}_{g}.
\]
Both summands are nonnegative continuous functions, so $\nabla R_{0}=0$
and $Q(R,R_{0})=0$ everywhere. Since $cR_1$ is parallel,
$\nabla R=0$. If $\Sigma_{k}(R|_{\mathfrak{g}})(p)>0$, then \eqref{eq:bochner-term-lower-bound}
forces $R_{0}(p)=0$. Since $R_0$ is parallel and $M$ is connected it follows that $R_{0}=0$
everywhere. Consequently $R|_{\mathfrak{g}}=c\operatorname{Id}$ and
$0<\Sigma_{k}(R|_{\mathfrak{g}})(p)=kc$, proving $c>0$. \end{proof}

\begin{remark}[Common algebraic sharpness] \label{rem:algebraic-sharpness}
The block in Lemma~\ref{lem:common-curvature-block} gives the sharpness
examples for all three settings. It first shows that the coefficient
6 in Proposition~\ref{prop:ordered} cannot be decreased: the first
two eigenvalues of $R_{*}$ are $-1,-1$, and \eqref{eq:common-model-support-action}
gives equality for both $l=1$ and $l=2$.

For the geometric algebraic implication, set 
\[
R_{c}=cR_{1}+R_{*},\qquad c>0.
\]
These are admissible algebraic Einstein curvature tensors of the respective
types. The spectrum on $\mathfrak{g}$ is 
\begin{center}
\begin{tabular}{|c|c|}
\hline
Eigenvalue & Multiplicity \\
\hline
$c-1$ & $2$   \\
\hline
$c$   & $N-3$ \\
\hline
$c+2$ & $1$   \\
\hline
\end{tabular}
\end{center}
The extra eigenvalues in the Kähler and quaternionic-Kähler cases
act trivially on $R_{*}$. Formula~\eqref{eq:common-model-total-action}
gives 
\begin{equation}
Q(R_{c},R_{*})=24\kappa c-72.\label{eq:common-sharpness-Q}
\end{equation}

For every admissible real $k\geq2$, not only for $k\leq N-1$, 
\begin{equation}
\Sigma_{k}(R_{c}|_{\mathfrak{g}})\geq kc-2.\label{eq:common-sharpness-spectral-bound}
\end{equation}
Indeed, the first two eigenvalues are $c-1$, and every remaining
eigenvalue is at least $c$. The same argument proves \eqref{eq:common-sharpness-spectral-bound}
for the full Kähler and quaternionic-Kähler operators, since their
additional eigenvalues are also at least $c$. Equality holds for
$2\leq k\leq N-1$, where only the two eigenvalues $c-1$ and copies
of $c$ enter the fractional sum. This includes the endpoint $k=2=N-1$
for $\mathfrak{su}(2)$.

In all three settings, 
\[
k_{*}=\frac{2\kappa}{3}\leq N-1,\quad c_{*}=\frac{3}{\kappa}=\frac{2}{k_{*}}.
\]
At $(k_{*},c_{*})$, both $Q(R_{c_{*}},R_{*})$ and the relevant $\Sigma_{k_{*}}$
are zero. Now fix any larger allowed threshold $k>k_{*}$. Choose
\[
\frac{2}{k}<c<\frac{3}{\kappa}.
\]
Then~\eqref{eq:common-sharpness-spectral-bound} gives $\Sigma_{k}>0$,
whereas~\eqref{eq:common-sharpness-Q} gives $Q(R_{c},R_{*})<0$.

Thus the three thresholds are sharp for the universal pointwise algebraic
implication from the stated spectral condition to nonnegativity of
the Bochner curvature term. \end{remark}

\appendix
\section{The Ricci tensor and harmonic curvature}
\label{sec:appendix-ricci}

The Einstein hypothesis in Theorem~1.1 makes the Ricci tensor parallel.
For metrics with harmonic curvature, a natural first question is whether
our curvature condition still makes the Bochner term of the Ricci tensor
nonnegative. We prove that this is the case in dimensions
$3\le n\le341$, and consequently obtain parallel Ricci curvature on
closed manifolds with harmonic curvature in this range. The argument
uses the fact that the tensor is the Ricci contraction of the same
curvature operator: the corresponding assertion for an arbitrary
symmetric two-tensor already fails in every dimension $n\ge5$.
We also construct algebraic counterexamples for the Ricci tensor in
every dimension $n\ge342$. Thus the dimension range is optimal for
this pointwise implication, but the examples do not settle the
geometric question in higher dimensions. Unlike the main body of the
paper, the geometric statement below does not assume that the metric
is Einstein.

We retain the curvature conventions of Section~2 and the notation
$\Sigma_\alpha$ from the introduction. Let
$R:\bigwedge^2 V \to\bigwedge^2 V$ be an algebraic
curvature operator, and let $\eta_1,\ldots,\eta_N$ be an orthonormal
eigenbasis with eigenvalues $\lambda_1\le\cdots\le\lambda_N$, where
$N=\binom n2$. For a symmetric two-tensor $h$, extend the notation
for the Bochner term by setting
\[
Q(R,h)=\sum_{i=1}^N\lambda_i|\eta_i \cdot h|^2
\]
and 
\[
(L\cdot h)(x,y)=-h(Lx,y)-h(x,Ly).
\]
Here $|h|$ is the tensor norm, equivalently the Hilbert-Schmidt norm
of the associated endomorphism of $\mathbb R^n$.
In an orthonormal eigenbasis $\{e_1,\ldots,e_n\}$ of $h$, with
corresponding eigenvalues $h_i$, we have 
\begin{equation}\label{eq:appendix-bochner}
\frac12 Q(R,h)=\sum_{i<j}K_{ij}(h_i-h_j)^2,
\end{equation}
where $K_{ij}=\langle R(e_i\wedge e_j),e_i\wedge e_j\rangle$. 
Indeed, the tensors $(e_i\wedge e_j)\cdot h$ are mutually
orthogonal and have squared norms $2(h_i-h_j)^2$.
In particular, if $r_i$ are the eigenvalues of $\operatorname{Ric}_R$,
then in an orthonormal basis diagonalizing the Ricci tensor, we have $r_i=\sum_{j\ne i}K_{ij}$ and
\begin{equation}
\label{eq:appendix-contraction}
\frac12 Q(R,\operatorname{Ric}_R)
=\sum_{i<j}K_{ij}(r_i-r_j)^2.
\end{equation}

\begin{theorem}
\label{thm:appendix-ricci}
\begin{enumerate}
\item For every $n\ge5$, there exist an $\frac{2(n-1)}{3}$-nonnegative
algebraic curvature operator $R$ and a trace-free symmetric
two-tensor $h$ such that $Q(R,h)<0$.
\item If $3\le n\le341$, every $\frac{2(n-1)}{3}$-nonnegative algebraic
curvature operator $R$ satisfies
$Q(R,\operatorname{Ric}_R)\ge0$.
\item For every $n\ge342$, there exists a $\frac{2(n-1)}{3}$-nonnegative
algebraic curvature operator $R$ such that
$Q(R,\operatorname{Ric}_R)<0$.
\end{enumerate}
\end{theorem}

\subsection{An arbitrary symmetric two-tensor}

\begin{proof}[Proof of Theorem~\ref{thm:appendix-ricci}(1)]
Take $h=\operatorname{diag}(1,-1,0,\ldots,0)$ and let $R$ be
diagonal in the basis $\{e_i\wedge e_j\}_{1 \leq i < j \leq n}$, with
\[
K_{12}=\frac{5-2n}{3}, \text{ and } 
K_{ij}=1 \text{ if }\{i,j\}\ne\{1,2\}.
\]
An operator diagonal in this simple-bivector basis satisfies the
first Bianchi identity. Its eigenvalues are $\frac{5-2n}{3}$ with multiplicity one and $1$
with multiplicity $N-1$, so
\[
\Sigma_{\frac{2(n-1)}{3}}(R)=\frac{5-2n}{3}+\left(\frac{2(n-1)}{3}-1\right)=0.
\]
On the other hand, for $n\geq 5$, we have
\[
\frac12 Q(R,h)=\frac{2(4-n)}3<0.
\]
\end{proof}

\subsection{The Ricci estimate}

The positive assertion follows from a more general bound. Set
\[
\beta=\left(\frac{10-\sqrt2}{7}\right)^2
=\frac{102-20\sqrt2}{49} \in \left(\frac{3}{2}, 2\right),
\]

\begin{proposition}
\label{prop:appendix-ricci}
Let $n\ge3$. If $R$ is an $\alpha$-nonnegative algebraic curvature
operator and
\[
1\le\alpha\le\frac n\beta
=\frac{51+10\sqrt2}{98}\,n,
\]
then $Q(R,\operatorname{Ric}_R)\ge0$.
\end{proposition}

We first record a consequence of Lemma~\ref{lem:prefix-comparison}. Suppose that
$1\le\alpha<N$ and $R$ is $\alpha$-nonnegative. In any orthonormal
basis $\{e_1,\ldots,e_n\}$, nonnegative weights $W_{ij}$ satisfying
\[
\alpha\max_{i<j}W_{ij}\le\sum_{i<j}W_{ij}
\]
obey
\begin{equation}
\label{eq:appendix-weight-comparison}
\sum_{i<j}K_{ij}W_{ij}\ge0.
\end{equation}
Indeed, let $W$ be the operator diagonalized by $e_i\wedge e_j$ with
entries $W_{ij}$. The assertion is immediate if $W=0$. Otherwise,
\[
0\le W\le\frac{\operatorname{tr}W}{\alpha}\,\mathrm{Id}.
\]
Its diagonal entries in an ordered eigenbasis of $R$ are nonnegative,
are bounded by $\operatorname{tr}W/\alpha$, and have sum
$\operatorname{tr}W$. Lemma~4.1 therefore gives
\[
\sum_{i<j}K_{ij}W_{ij}
=\operatorname{tr}(RW)
\ge\frac{\operatorname{tr}W}{\alpha}\Sigma_\alpha(R)\ge0.
\]

\begin{proof}[Proof of Proposition~\ref{prop:appendix-ricci}]
Write $\operatorname{Ric}_R^\circ$ for the traceless Ricci tensor
and $r_i^\circ=r_i-\frac1n\operatorname{tr}(\operatorname{Ric}_R)$ for its eigenvalues in the chosen Ricci eigenbasis.
If $\operatorname{Ric}_R^\circ=0$, there is nothing to prove.
By positive rescaling and relabeling, we may assume $r_1=\max_i r_i$, $r_2=\min_i r_i$, and $r_1-r_2=2$.
Then $|\operatorname{Ric}_R^\circ|^2\ge2$ and
\begin{equation}
\label{eq:appendix-total}
    \sum_{i<j}(r_i-r_j)^2=n|\operatorname{Ric}_R^\circ|^2.
\end{equation}

If $4\le\beta|\operatorname{Ric}_R^\circ|^2$, the comparison
\eqref{eq:appendix-weight-comparison} applies directly to the
weights $(r_i-r_j)^2$, since $\alpha\beta\le n$.
We may therefore assume
\[
D=4-\beta|\operatorname{Ric}_R^\circ|^2>0.
\]
In particular, $D<1$.

For $j\ge3$, put $\ell_j=r_1-r_j\in[0,2]$.
Choose $s>0$ so that
\begin{equation}
\label{eq:appendix-s}
    \sum_{j=3}^n(s-\ell_j^2)_+=D,
\end{equation}
which is possible by continuity. Define $u_1=s$, $u_2=D-s$, $u_j=-(s-\ell_j^2)_+$ for $j \geq 3$, and 
\[
W_{ij}=(r_i-r_j)^2-u_i-u_j.
\]
Since $\sum_i u_i=0$, the modified weights have the same total
\begin{equation}
\label{eq:appendix-modified-total}
    \sum_{i<j}W_{ij}=n|\operatorname{Ric}_R^\circ|^2.
\end{equation}
We will show that
\begin{equation}\label{eq W_ij bound}
    0\le W_{ij}\le\beta|\operatorname{Ric}_R^\circ|^2,
\end{equation}
and 
\[
\sum_i u_i r_i^\circ\ge0.
\]
The first assertion is the spectral comparison, while the
second controls the change in the Bochner term.

\emph{Step 1: bound the modified weights.}
The three entries $r_1^\circ,r_2^\circ,r_j^\circ$ give
\[
    |\operatorname{Ric}_R^\circ|^2
    \ge\frac{4+\ell_j^2+(2-\ell_j)^2}{3}.
\]
Since $\beta>3/2$, this implies $ D\le\ell_j(2-\ell_j)$ and 
\[
    \ell_j^2\le\frac32|\operatorname{Ric}_R^\circ|^2
    <\beta|\operatorname{Ric}_R^\circ|^2.
\]
Fix an index $k\ge3$ minimizing $\ell_k$.
By \eqref{eq:appendix-s}, we have $\ell_k^2<s$ and
$s-\ell_k^2\le D$. Hence
\begin{equation}
\label{eq:appendix-s-bound}
    s\le\ell_k^2+D\le2\ell_k\le2\ell_j, \text{ for } j \geq 3.
\end{equation}
It follows that 
$$W_{12}=4-D=\beta|\operatorname{Ric}_R^\circ|^2$$ 
and 
$$W_{1j}=(\ell_j^2-s)_+
  \leq \beta|\operatorname{Ric}_R^\circ|^2.$$

Also,
\[
\begin{aligned}
    W_{2j}+D
    &=(2-\ell_j)^2+s+(s-\ell_j^2)_+\\
    &\le(2-\ell_j)^2+2\ell_j+(2\ell_j-\ell_j^2) \\
    & =4.
\end{aligned}
\]
For the lower bound, $s>\ell_k^2$ gives
\[
    W_{2j}>\ell_k^2+(2-\ell_j)^2-D.
\]
When $j=k$, the right-hand side is at least $2-D>0$.
When $j\ne k$, the weighted Cauchy-Schwarz inequality
$x^2+\sqrt2\,y^2\ge(2-\sqrt2)(x+y)^2$ gives
\[
    \ell_k^2+(2-\ell_j)^2
    +\sqrt2\,|\operatorname{Ric}_R^\circ|^2
    \ge2\bigl((r_1^\circ)^2+(r_2^\circ)^2\bigr)
    \ge4.
\]
Since $\beta>\sqrt2$, this again yields $W_{2j}>0$.
Finally, for $3\le j<k\le n$, the disjoint pairs $(1,2)$ and
$(j,k)$ give 
$$(r_j-r_k)^2\le2|\operatorname{Ric}_R^\circ|^2-4.$$
Since $-u_j-u_k\le D$, it follows that
\[
    0\le W_{jk}
    \le2|\operatorname{Ric}_R^\circ|^2-4+D
    =(2-\beta)|\operatorname{Ric}_R^\circ|^2
    \le\beta|\operatorname{Ric}_R^\circ|^2.
\]
This proves the required bounds in \eqref{eq W_ij bound}.

\emph{Step 2: control the correction.}
Set
\[
    C=2-\beta\bigl((r_1^\circ-1)^2+1\bigr).
\]
Since $r_2^\circ=r_1^\circ-2$, we have
\[
    2C=4-\beta\bigl((r_1^\circ)^2+(r_2^\circ)^2\bigr)
    \ge D>0.
\]
The choice of $\beta$ gives
\[
    2-\beta=4\left(\sqrt\beta-1\right)\left(3-2\sqrt\beta\right),
\]
with both factors on the right positive.
Using $C+\beta(r_1^\circ-1)^2=2-\beta$, Cauchy-Schwarz
followed by the inequality $2\sqrt{ab} \leq a+b$ for $a,b \geq 0$ yields, for $0\le\ell\le2$,
\[
\begin{aligned}
    \sqrt{C(2-\ell)}-\sqrt\beta\,(r_1^\circ-1)
    &\le\sqrt{(2-\beta)(3-\ell)}\\
    &\le(\sqrt\beta-1)(3-\ell)+(3-2\sqrt\beta).
\end{aligned}
\]
Consequently,
\begin{equation}
\label{eq:appendix-scalar}
    \sqrt{C(2-\ell)}
    \le\ell+\sqrt\beta\,(r_1^\circ-\ell).
\end{equation}
If $u_j<0$, then $\ell_j^2-u_j=s$ and
$r_1^\circ-\ell_j=r_j^\circ$. Another application of
Cauchy-Schwarz therefore gives
\[
\begin{aligned}
    C(2-\ell_j)(-u_j)
    &\le(\ell_j+\sqrt\beta\,r_j^\circ)^2(-u_j)\\
    &\le(\ell_j^2-u_j)
          \bigl(-u_j+\beta(r_j^\circ)^2\bigr)\\
    &=s\bigl(-u_j+\beta(r_j^\circ)^2\bigr).
\end{aligned}
\]
The same inequality is immediate when $u_j=0$.
Summing over $j\ge3$ gives
\[
    C\sum_{j=3}^n(2-\ell_j)(-u_j)
    \le s\left(D+\beta\sum_{j=3}^n(r_j^\circ)^2\right)
    =2sC.
\]
Since $C>0$, we conclude that
\begin{equation}
\label{eq:appendix-correction}
    \sum_i u_i r_i^\circ
    =2s-\sum_{j=3}^n(2-\ell_j)(-u_j)\ge0.
\end{equation}

To finish, Step~1 and \eqref{eq:appendix-modified-total} imply
\[
    \alpha\max_{i<j}W_{ij}
    \le\alpha\beta|\operatorname{Ric}_R^\circ|^2
    \le n|\operatorname{Ric}_R^\circ|^2
    =\sum_{i<j}W_{ij}.
\]
Applying \eqref{eq:appendix-weight-comparison} and using
$\sum_i u_i=0$, we obtain
\[
\begin{aligned}
    \frac12 Q(R,\operatorname{Ric}_R)
    &=\sum_{i<j}K_{ij}W_{ij}+\sum_i u_i r_i\\
    &=\sum_{i<j}K_{ij}W_{ij}+\sum_i u_i r_i^\circ
    \ge0.
\end{aligned}
\]
\end{proof}

\begin{proof}[Proof of Theorem~\ref{thm:appendix-ricci}(2)]
For $\alpha=\frac{2(n-1)}{3}$, the condition $\alpha\beta\le n$ is
equivalent to
\[
    n\le\frac{2\beta}{2\beta-3}=172+120\sqrt2.
\]
Since $341<172+120\sqrt2<342$, the assertion follows from
Proposition~\ref{prop:appendix-ricci}.
\end{proof}

\subsection{Algebraic counterexamples in higher dimensions}
\label{subsec:appendix-counterexamples}

\begin{proof}[Proof of Theorem~\ref{thm:appendix-ricci}(3)]
Let $n\ge342$ and set
\[
\alpha=\frac{2(n-1)}3
 \text{ and } c=\frac{35}{18(n-1)}.
\]
Define $R$ to be diagonal in the basis $\{e_i\wedge e_j\}_{i < j}$, with
\begin{equation}
\label{eq:appendix-counterexample}
K_{ij}=
\begin{cases}
c-\dfrac{35}{27}, & (i,j)=(1,2),\\[4pt]
c+\dfrac{2}{17}, & i=1,\quad 3\le j\le19,\\[4pt]
c+\dfrac{2}{27(n-20)}, & 20\le i<j\le n,\\[4pt]
c, & \text{otherwise}.
\end{cases}
\end{equation}
As before, being diagonalized by this basis ensures the first Bianchi
identity. The operator has one negative eigenvalue $c-\frac{35}{27}$;
the eigenvalue $c$ has multiplicity $19n-208>\alpha$, and all
remaining eigenvalues exceed $c$. Consequently,
\[
\Sigma_\alpha(R)=\alpha c-\frac{35}{27}=0.
\]
The Ricci tensor is diagonal, and its eigenvalues have mean
$109/54$. Writing $d_i=r_i-109/54$, we find
\[
d_1=\frac{17}{27},\quad
 d_2=-\frac{37}{27},\quad
 d_3=\cdots=d_{19}=\frac{20}{459},\quad
 d_{20}=\cdots=d_n=0.
\]
In particular,
\[
\sum_i d_i^2=\frac{28586}{12393},
\quad d_1-d_2=2,
\quad d_1-d_j=\frac{269}{459}\text{ for } 3\le j\le19.
\]
The change to $K_{ij}$ for $20\le i<j\le n$ contributes nothing
to the Bochner term. Using $\sum_i d_i=0$, we therefore obtain
\begin{align*}
\frac12 Q(R,\operatorname{Ric}_R)
&=nc\frac{28586}{12393}
 -\frac{35}{27}\cdot4
 +2\left(\frac{269}{459}\right)^2\\
&=-\frac{24947n-8529282}{1896129(n-1)}<0.
\end{align*}
The numerator is $2592$ at $n=342$ and increases with $n$.
\end{proof}

\subsection{Manifolds with harmonic curvature}

\begin{corollary}
\label{cor:appendix-harmonic}
Let $(M^n,g)$ be a closed, connected Riemannian manifold with
harmonic curvature and $3\le n\le341$. If its curvature operator
is $\frac{2(n-1)}{3}$-nonnegative at every point, then
$\nabla\operatorname{Ric}=0$ and $(M,g)$ is locally symmetric.
\end{corollary}

\begin{proof}
Harmonic curvature means that the curvature tensor is
divergence-free. The second Bianchi identity implies that the
Ricci tensor is Codazzi, and the contracted Bianchi identity then
shows that its trace is constant. The integrated Bochner identity for the Ricci tensor is
\[
0=\int_M\left(
 |\nabla\operatorname{Ric}|^2
 +\frac12 Q(R,\operatorname{Ric})
 \right)\,d\operatorname{vol}_g.
\]
By Theorem~\ref{thm:appendix-ricci}(2), both summands are
nonnegative. It follows that $\nabla\operatorname{Ric}=0$.

If $g$ is Einstein and $n\ge4$, Theorem 1.1 applies. If $n=3$ and $g$ is Einstein, it has constant sectional curvature. Otherwise, choose a proper parallel eigensubbundle $E$ of the Ricci endomorphism, of real rank $p$. The local product splitting $TM=E\oplus E^\perp$ implies that the curvature operator annihilates mixed bivectors. Consequently,
$$\dim (\ker R) \ge p(n-p)\ge n-1>\frac{2(n-1)}3.$$
A negative eigenvalue together with these zero eigenvalues would force $\Sigma_{\frac{2(n-1)}3}(R)<0$. Thus $R\ge0$. Since the curvature tensor is harmonic, its integrated Bochner identity gives $\nabla R=0$.
\end{proof}

The examples in Subsection~\ref{subsec:appendix-counterexamples}
are algebraic curvature operators, not metrics on closed manifolds with harmonic
curvature. They establish the optimal dimension range for the
pointwise implication in Theorem~\ref{thm:appendix-ricci}(2), but
do not establish failure of Corollary~\ref{cor:appendix-harmonic}
in dimensions $n\ge342$.

\section{K\"ahler manifolds with harmonic curvature}\label{sec:appendix-kahler-harmonic}

Theorem~\ref{thm:kahler} extends to K\"ahler metrics with harmonic curvature.
The extension uses the classical theorem of Matsushima \cite{Matsushima1972}
that such metrics have parallel Ricci tensor. In the non-Einstein case, the resulting parallel splitting supplies enough zero eigenvalues to reduce the spectral assumption to nonnegativity of the primitive curvature operator.

\begin{theorem}
Let $(M^{2n},g,J)$ be a closed, connected K\"ahler manifold of
complex dimension $n\ge2$ with harmonic curvature,
$\operatorname{div}R=0$. Let $\pi_0$ denote orthogonal projection
onto $\Lambda^{1,1}_0$, and let
\[
    K_0=\pi_0 \circ R\big|_{\Lambda^{1,1}_0}
\]
be the primitive curvature compression.
If $K_0$ is $\frac{2(n+1)}3$-nonnegative at every point, then
$(M,g)$ is locally symmetric. 
\end{theorem}

\begin{proof}
Set $k=\frac{2(n+1)}{3}$. By Matsushima’s theorem \cite{Matsushima1972},
$\nabla\operatorname{Ric}=0$. If $g$ is Einstein, the conclusions
follow from Theorem~1.2 and the strict assertion of Lemma~4.5.

Suppose that $g$ is not Einstein. Choose a proper eigensubbundle
$E$ of the Ricci endomorphism. It is parallel and $J$-invariant,
with complex rank $1\le p\le n-1$. The local product splitting
$TM=E\oplus E^\perp$ implies that $R$ annihilates the mixed
skew-Hermitian endomorphisms. These are primitive and form a
real $2p(n-p)$-dimensional subspace. Hence
\[
    \dim\ker K_0
    \ge 2p(n-p)
    \ge 2(n-1)
    \ge k.
\]
If $K_0$ had a negative eigenvalue, this nullity would force
$\Sigma_k(K_0)<0$. Thus the spectral assumption gives $K_0 \geq 0$ and $ \Sigma_k(K_0)=0$.  
In particular, the strict hypothesis cannot occur in this case.

Let $\eta_1,\ldots,\eta_{n^2-1}$ be an orthonormal eigenbasis
of $K_0$, with eigenvalues $\lambda_i\ge0$, and extend it to
an orthonormal basis of $\mathfrak u(n)$ by
$\eta_0=\frac{\omega}{\sqrt n}$.
The central direction acts trivially on the K\"ahler curvature
tensor, so $\eta_0\cdot R=0$. Consequently, all central and
mixed central--primitive terms vanish in the curvature
contraction, leaving
\[
    Q(R,R)
    =\sum_{i=1}^{n^2-1}\lambda_i\|\eta_i\cdot R\|^2
    \ge0.
\]
This identity uses only the compression $K_0$ and does not
require $\Lambda^{1,1}_0$ to be invariant under $R$.

Since $R$ is harmonic, its integrated Bochner identity
\cite[Proposition~1.4 and Corollary~1.5]{PW22} gives
\[
    0=\int_M
    \left(
        \|\nabla R\|^2+\frac12 Q(R,R)
    \right)\,d\operatorname{vol}_g.
\]
Both terms are nonnegative, so $\nabla R=0$.
\end{proof}

\bibliographystyle{alpha}
\bibliography{referencesv2}

\newcommand{\etalchar}[1]{$^{#1}$}
\begin{thebibliography}{BNP{\etalchar{+}}26}

\bibitem[Ale68]{Alekseevskii1968}
D.~V. Alekseevskii.
\newblock {Riemannian} spaces with exceptional holonomy groups.
\newblock {\em Funct. Anal. Appl.}, 2(2):97--105, 1968.

\bibitem[Ber66]{Berger1966}
Marcel Berger.
\newblock Sur les vari\'{e}t\'{e}s d'{Einstein} compactes.
\newblock In {\em Comptes rendus de la {IIIe} R\'{e}union du Groupement des Math\'{e}maticiens d'Expression Latine ({Namur}, 1965)}, pages 35--55. Librairie universitaire, 1966.

\bibitem[Bes87]{Besse1987}
Arthur~L. Besse.
\newblock {\em {Einstein} manifolds}, volume~10 of {\em Ergebnisse der Mathematik und ihrer Grenzgebiete (3)}.
\newblock Springer-Verlag, Berlin, 1987.

\bibitem[BNP{\etalchar{+}}26]{BNPSW2026}
Kyle Broder, Jan Nienhaus, Peter Petersen, James Stanfield, and Matthias Wink.
\newblock Vanishing theorems for {Hodge} numbers and the {Calabi} curvature operator.
\newblock {\em Adv. Math.}, 500:111077, 2026.

\bibitem[Boc49]{Bochner1949}
Salomon Bochner.
\newblock Curvature and {Betti} numbers. {II}.
\newblock {\em Ann. of Math. (2)}, 50(1):77--93, 1949.

\bibitem[Bre10]{Brendle2010}
Simon Brendle.
\newblock {Einstein} manifolds with nonnegative isotropic curvature are locally symmetric.
\newblock {\em Duke Math. J.}, 151(1):1--21, 2010.

\bibitem[BS25]{BrendleSemmelmann2025}
Simon Brendle and Uwe Semmelmann.
\newblock {Quaternionic-K\"ahler} manifolds with nonnegative sectional curvature.
\newblock {\em arXiv:2506.22181}, 2025.

\bibitem[BW08]{BW08}
Christoph B{\"o}hm and Burkhard Wilking.
\newblock Manifolds with positive curvature operators are space forms.
\newblock {\em Ann. of Math. (2)}, 167(3):1079--1097, 2008.

\bibitem[CMR24]{ColomboMarianiRigoli2024}
Giulio Colombo, Marco Mariani, and Marco Rigoli.
\newblock {Tachibana}-type theorems on complete manifolds.
\newblock {\em Ann. Sc. Norm. Super. Pisa Cl. Sci. (5)}, 25(2):1033--1083, 2024.

\bibitem[Gra77]{Gray1977}
Alfred Gray.
\newblock Compact {K{\"a}hler} manifolds with nonnegative sectional curvature.
\newblock {\em Invent. Math.}, 41:33--43, 1977.

\bibitem[LS94]{LeBrunSalamon1994}
Claude LeBrun and Simon Salamon.
\newblock Strong rigidity of positive quaternion-{K{\"a}hler} manifolds.
\newblock {\em Invent. Math.}, 118(1):109--132, 1994.

\bibitem[Mat72]{Matsushima1972}
Yozo Matsushima.
\newblock Remarks on {K{\"a}hler--Einstein} manifolds.
\newblock {\em Nagoya Math. J.}, 46:161--173, 1972.

\bibitem[MW93]{MM93}
Mario~J. Micallef and McKenzie~Y. Wang.
\newblock Metrics with nonnegative isotropic curvature.
\newblock {\em Duke Math. J.}, 72(3):649--672, 1993.

\bibitem[MZ86]{MokZhong1986}
Ngaiming Mok and Jia-Qing Zhong.
\newblock Curvature characterization of compact {Hermitian} symmetric spaces.
\newblock {\em J. Differential Geom.}, 23(1):15--67, 1986.

\bibitem[PW21a]{PW21}
Peter Petersen and Matthias Wink.
\newblock New curvature conditions for the {Bochner} technique.
\newblock {\em Invent. Math.}, 224(1):33--54, 2021.

\bibitem[PW21b]{PW21Crelle}
Peter Petersen and Matthias Wink.
\newblock Vanishing and estimation results for {Hodge} numbers.
\newblock {\em J. Reine Angew. Math.}, 780:197--219, 2021.

\bibitem[PW22]{PW22}
Peter Petersen and Matthias Wink.
\newblock {Tachibana}-type theorems and special holonomy.
\newblock {\em Ann. Global Anal. Geom.}, 61(4):847--868, 2022.

\bibitem[Sal82]{Salamon1982}
Simon Salamon.
\newblock Quaternionic {K{\"a}hler} manifolds.
\newblock {\em Invent. Math.}, 67(1):143--171, 1982.

\bibitem[Tac74]{Tachibana1974}
Shun-ichi Tachibana.
\newblock A theorem on {Riemannian} manifolds of positive curvature operator.
\newblock {\em Proc. Japan Acad.}, 50(4):301--302, 1974.

\end{thebibliography}

\end{document}